\documentclass[11pt, oneside,reqno]{amsart}
\usepackage{amsfonts,amsmath,mathabx,fullpage,amssymb,amsthm,tikz,mathtools,lmodern,paralist}
\usepackage{mathrsfs,hyperref,stmaryrd,enumitem,tikz-cd,shuffle}

\usepackage{color}

\newcommand{\Zz}{\mathbb{Z}}

\newcommand{\PS}{\mathrm{PS}}

\newcommand{\BB}{\mathcal{B}}

\newcommand{\II}{\mathcal{I}}

\newcommand{\PP}{\mathcal{P}}
\newcommand{\OO}{\mathcal{O}}

\newcommand{\LL}{{\mathcal L}}

\numberwithin{equation}{section}

\newtheorem{theorem}{Theorem}[section]
\newtheorem{proposition}[theorem]{Proposition}
\newtheorem{lemma}[theorem]{Lemma}
\newtheorem{corollary}[theorem]{Corollary}

\theoremstyle{definition}
\newtheorem{definition}[theorem]{Definition}
\newtheorem{example}[theorem]{Example}
\newtheorem{remark}[theorem]{Remark}

\newtheorem{problem}[theorem]{Problem}
\newtheorem{question}[theorem]{Question}

\newcommand{\dllin}{\ar@{-}[dl]}
\newcommand{\drlin}{\ar@{-}[dr]}
\newcommand{\dlin}{\ar@{-}[d]}

\begin{document}
\title{Finiteness theorems for matroid complexes with prescribed topology}
\author{Federico Castillo}
\address{Max Planck Institute for Mathematics in the Science}
\email{castillo@mis.mpg.de}
\author{Jos\'e Alejandro Samper}
\address{Departamento de Matemáticas, Pontificia Universidad Católica de Chile}
\email{jsamper@mat.uc.cl}
\date{\today}
\maketitle
\begin{abstract}
There are finitely many simplicial complexes (up to isomorphism) with a given number of vertices. Translating this fact to the language of $h$-vectors, there are finitely many simplicial complexes of bounded dimension with $h_1=k$ for any natural number $k$. In this paper we study the question at the other end of the $h$-vector: Are there only finitely many $(d-1)$-dimensional simplicial complexes with $h_d=k$ for any given $k$? The answer is no if we consider general complexes, but we focus on three cases coming from matroids: (i) independence complexes, (ii) broken circuit complexes, and (iii) order complexes of geometric lattices. Surpsingly, the answer is yes in all three cases.  
\end{abstract}

\section{Introduction}

This paper aims to present a new approach to the study of matroids from the perspective of the topology of various simplicial complexes. In the survey \cite{bjorner}, Bj\"orner presented the story of three complexes associated to a matroid: the independence complex, the broken circuit complex, and the order complex of its lattice of flats. The main idea is to parametrize these complexes by their homotopy type.

To understand the various aspects of the topology of the aforementioned complexes we start by recalling that they are all shellable and hence homotopy equivalent to the wedge of some finite number of spheres all of the same dimension as the whole complex. The homotopy type is then completely determined by two parameters, the dimension and the Euler characteristic.

The corresponding $h$-numbers, and their equivalent relatives $f$-numbers, have been extensively studied in the literature and are the subject of widely celebrated new results and old conjectures. For instance, the recent resolution of the Rota-Herron-Welsh conjecture by Adiprasito, Huh and Katz \cite{AHK} can be interpreted as a set of inequalities on $f$-vectors of broken circuit complexes. In another recent breakthrough Ardila, Dehnham and Huh \cite{Ardila-Slides} managed to generalize results of \cite{huhh} and prove that the $h$-vector of any broken circuit complex, and hence of any independence complex, is a log concave sequence. Other recent breakthroughs include the resolution of the strongest version of Mason's conjecture for $f$-vectors of independence complexes done independently by Br\"and\'en-Huh \cite{BH} and Anari-Lui-Oreis Gharan-Vinzant \cite{ALOV}.

From the work of Chari \cite{Chari} (for independence complexes), Nyman and Swartz \cite{EdNym} (for order complexes of geometric lattices), and Juhnke-Kubitzke and Le \cite{Martina} (for broken circuit complexes) we now know that the $h$-vector in all these cases is \emph{flawless}. In terms of the entries it says that if $h= (h_0, \dots, h_s)$  is the $h$-vector of a complex, with $h_s\neq 0$  and  $\delta = \lfloor \frac s2 \rfloor$, then $h_0 \le h_1 \le \dots \le h_\delta$ and $h_i \le h_{s-i}$ for $i\le \delta$. 

It is known that the $h$-vector of any simplicial complex remains fixed after adding cone points: the operation adds as many zeros to the right end as the number of added cone points. The largest index $s$ such that $h_s\neq 0$ equals the size of any maximal face if the complex is shellable and not contractible. For all the complexes studied here, being contractible is equivalent to being a cone, so the zeros at the right end are of no major consequence and we can assume that the complex is not contractible and $s=d$, where $d-1$ is the dimension of the complex. 

If $i<d$ and the $h$-vector is flawless, then $h_i \ge h_1= f_0-d$, where $f_0$ is the number of vertices. It follows that, after fixing $k$ and $d$, the number of (isomorphism types of) complexes of rank $d$ with $h_i=k$ and no cone vertices is finite. This is however, far away from the case if we consider $h_d$ instead: the $g$-theorem \cite[Theorem 1.1 Section III]{greenbook} implies that the $h$-vector of the boundary of any $(d-1)$-dimensional simplicial polytope is flawless and has $h_d=1$. 

Interestingly for the three families of complexes studied here, the restriction for $h_d$ still implies finiteness. We now summarize the results.

\subsection{Independence complexes}

Perhaps the most intriguing conjecture about matroid $h$-vectors is due to Stanley~\cite{Stanley77}. It posits that the $h$-vector of a matroid independence complex is a pure $O$-sequence. This means that, given one such $h$-vector $(h_0, \dots, h_d$), there is a finite collection of monomials $\mathcal S$ satisfying the following three properties: \begin{enumerate}\item[i.]  $\mathcal S$ is closed under divisibility, 
\item [ii.] $\mathcal S$ has $h_i$ monomials of degree $i$, and
\item[iii.] Every monomial in $\mathcal S$ divides a monomial of degree $d$ in $\mathcal S$. 
\end{enumerate}
Among these three conditions, the third is the toughest to achieve. It follows from the results in \cite{Stanley77} that $\mathcal S$ can be constructed satisfying the other two conditions. The proof yields a collection of inequalities satisfied by the entries of the $h$-vectors. However $h$-vector families are much smaller in all of our three cases, than the family of $h$-vectors satisfying conditions [i.] and [ii.], i.e. the class Cohen-Macaulay simplicial complexes. The third property is perhaps an attempt to capture this for matroid independence complex. It is the combinatorial analogue of a result in the realm of commutative algebra: the Artinian reduction of the Stanley-Reisner ring (over any field) of the independence complex of a matroid is level \cite[Theorem 3.4 Section III]{greenbook}. 

Among enumerative consequences of [iii.] is that $h_1$ is bounded above in terms of $h_d$: each monomial of degree one divide one monomial of degree $d$, thus $h_1 \le dh_d$. 
This in turn, would yield a finiteness result that is the starting point of this paper: we don't need Stanley's conjecture to obtain much better bounds than the prediction of this conjecture. The consequences of such a statement are strong. 
\begin{theorem}\label{thm:main} Let $d,k$ be positive integers. There are finitely many isomorphism classes of loop free rank $d$ matroids $M$ whose independence complex satisfies $h_d(\II(M))=k$. 
\end{theorem}
This surprisingly looking result is an overlooked consequence of several results that exist in the literature, some dating back to 1980. Even more remarkable is Corollary \ref{cor:cosimple}: the rank condition can be dropped by restricting to the class of cosimple matroids. 

Theorem \ref{thm:main} implies that there are upper bounds on all $h$-numbers in terms of $h_d$. On the other hand, lower bounds exist from the fact that the $h$-vector is an $O$-sequence. Thus it seems reasonable to launch a program to understand extremal matroids for upper and lower bounds for matroid independence complexes with fixed rank and topology. 

Notice that a similar program for simplicial polytopes in terms of vertices and dimension has been widely successful: it leads to the stories of neighborly and stacked polytopes \cite[Theorems 2 and 3]{bellasurvey}. On the other hand its counterpart for matroids based in rank and the number of vertices does not say much. For example, all upper bounds are achieved trivially by uniform matroid. 

In contrast, by using the top $h$-number instead, the upper bound analogue has a non-trivial maximizer and restricting to the classes of simple and connected matroids changes the problem drastically. For lower bounds, uniform matroids are entrywise minimizers but only for certain values of $h_d$. 

Another natural path to follow is trying to estimate the size of the set $\Psi_{d,k}$ of all isomorphism classes of loopless matroids of rank $d$ with $h_d=k$. It is a priori not clear that such a set is not empty, but we provide several examples in each class. Furthermore, we provide non-trivial upper and lower bounds for the cardinality of $\Psi_{d,1}$. In particular, we extend a result of Chari, who showed that $|\Psi_{d,1}| = p(d)$, the number of integer partitions of $d$. 

\begin{theorem}\label{thm:howmany}
\label{thm:ref} Let $d,k>0$ and let $T_{d,k}$ be the number of matroids of rank at most $d$ with at most $k$ bases. Then \[2^dkT_{d,k}\ge |\Psi_{d,k}| \ge |\Psi_{d,1}| = p(d).\]
\end{theorem}

The bounds above are far from tight. Nonetheless we expect the asymptotics to be close to the upper bound. It is not even clear that the cardinality of $\Psi_{d,k}$ increases as $d$ or $k$ increase. Furthermore, restricting to the subset $\Sigma_{d,k}$ of $\Psi_{d,k}$ that consists of isomorphism classes of simple matroids one observes the following:  $|\Sigma_{2,1}| = 1 > 0 = |\Sigma_{2,2}|$. Hence a wilder behavior in the case of simple matroids is expected.  

\subsection{Broken circuit complexes}

A natural question that follows after studying independence complexes is that of broken circuit complexes. They arise naturally in the study of hyperplane arrangements and are a meaningful generalization of matroids: every matroid is a reduced broken circuit complex. The following is a reinterpretation of \cite[Theorem 5.4]{superEd}.

\begin{theorem}\label{thm:nbc}
Let $d,k$ be positive integers. The number of isomorphism classes of simple connected, rank $d$ ordered matroids $M$ whose reduced broken circuit complex satisfies $h_{d-1}(\overline{BC_<(M)})=k$ is finite. 
\end{theorem}


It is known that $h$-vectors of broken circuit complexes properly contain the $h$-vectors of matroids. This is seen through specific examples. However, to the best of our knowledge, all numerical inequalities currently known to be true for independence complexes are also known to be true for broken circuit complexes. For instance, recently Ardila-Denham-Huh proved that the $h$-vectors of broken circuit complexes are log concave \cite[Theorem 1.4]{lagrangian}.

Additionally we provide a theorem that we hope is related to the finiteness results, which turns out to be of independent interest. It provides evidence that broken circuit complexes play an important role in the theory of quasi-matroidal classes \cite{QS-stuff}. 
\begin{theorem}\label{orderId} If $(M,<)$ is an ordered loopless matroid, then the nbc bases form an order ideal of $\text{Int}_<(M)$. 
\end{theorem}
$\text{Int}_<(M)$ is a poset on the set of bases of $M$ defined by Las Vergnas by means of internal activities.

\subsection{Geometric lattices}
Interest in geometric lattices has flourished significantly in the last two decades due to their connection with tropical geometry. They are connected to tropical linear spaces via the Bergman fan of $M$.  Intersecting the fan with a unit sphere results in a geometric realization of the order complex of the lattice of flats of $M$. See for instance \cite{Ardila-Klivans}. It is also crucial in the study of the Chow ring of a matroid and its Hodge structure \cite{AHK}. Even more, Huh and Wang \cite{Huh-Wang} recently proved Dowling's top heavy conjecture for representable geometric lattices: a theorem on numerical invariants of the lattice, by studying again elements of Hodge theory. It is therefore desirable to get a better grasp of aforementioned invariants from a different point of view as a way to complement the new results.

Hidden in one of the exercises in \cite[Problem 100.(d) Ch. 3 ]{EC1} is a reformulation of the following: the number of isomorphism classes of simple, loop and coloop free matroids whose geometric lattice is homotopy equivalent to a wedge of $k$ spheres (independently of dimension!) is finite. This is much stronger than the result for independence complexes and can be expressed in terms of Euler characteristics, M\"obius functions or the top non-zero $h$-number of the order complex of the proper part of the lattice. Even though the result is stated in Stanley's book, there seems to be no published proof.

\begin{theorem}\label{thm:flats}
Let $d,k$ be positive integers. The number of isomorphism classes of simple matroids $M$ of rank $d$ whose geometric lattice, $\LL(M)$, satisfies $|\mu(\LL(M))|=k$ is finite. Furthermore if we restrict to coloop free matroids, we can drop the rank condition.
\end{theorem}
One proof of the above mentioned result and a sketch of weaker rank dependent result are included, mainly because their flavor is similar to that of independence complexes and it looks like the techniques can be improved to obtain additional structural properties that complement and deepen Stanley's result. 

In the last section we pose a number of new questions, including an upper and lower bound program, and a more detailed study of the topology of cosimple matroids. In these cases many the relevant parameters to stratify the resulting classes of objects are non standard, but seem natural choices given our results. 

This article is organized as follows. Section 2 recalls basic definitions and concepts needed in the paper. Section 3 contains the results about independence complexes. Section 4 discusses broken circuit complexes, section 5 geometric lattices, and section 6 poses questions and future directions of research. \\

\section{Definitions and notation}

This section is devoted to defining, summarizing and relating various aspects of matroid theory that appear in the arguments of this paper.

\subsection{Simplicial complexes} A simplicial complex $\Delta$ is a collection of subsets of a finite set $E$ that is closed under inclusion. Any simplicial complex admits a geometric realization, a topological space whose different aspects (geometric and topological) encode the information about the complex. The topology of a simplicial complex refers to the topology of its geometric realization. Throughout this paper we use reduced simplicial homology with rational coefficients. We refer the readers to \cite{greenbook} for details and undefined terminology.

Elements of a simplicial complex $\Delta$ are called faces. The complex $\Delta$ is said to be \emph{pure} if all its maximal faces have the same cardinality. For a subset $A$ of the base set of $\Delta$ (also known as the ground set or vertex set), let $\Delta|_A$ be the complex consisting of the faces of $\Delta$ contained in $A$. The complex $\Delta|_A$ is said to be an \emph{induced} subcomplex of $\Delta$. The \emph{dimension} of a face of a complex is one less than its cardinality and the dimension of a complex is the maximal dimension of its faces. The $f$-vector $(f_{-1}, f_0, f_1, \dots , f_{d-1})$ of a simplicial complex $\Delta$ is the enumerator of faces by dimension, i.e., $f_k$ denotes the number of $k$-dimensional faces of $\Delta$. We encode the $f$-vector by its generating function, the \emph{face enumerator} $f(\Delta,t)=\sum_{i=0}^{d} f_{i-1}t^{d-i}$. The \emph{reduced euler characteristic} is  $\Tilde{\chi}(\Delta):=-f_{-1}+f_0-f_1+\cdots=(-1)^{d-1}f(\Delta,-1)$.

The $h$-vector of a complex $\Delta$ is a vector that carries the exact same information as the $f$-vector. It is defined as the following coefficients $f(\Delta,(t-1))=\sum_{i=0}^{d} f_{i-1}(t-1)^{d-i}=\sum_{k=0}^d h_kt^{d-k}$. 
\begin{remark}\label{rem:hpositive}
Notice that by replacing $t\to t+1$ in the previous equation we obtain the relation $\sum_{i=0}^{d} f_{i-1}t^{d-i}=\sum_{k=0}^d h_k(t+1)^{d-k}$ which shows that the entries of the $f$-vector are a positive combination of the entries of the $h$-vector.
\end{remark}

We say a complex $\Delta$ is an iterated cone if there exists a set $\gamma=\{v_1,\cdots,v_k\}$ such that $\gamma$ is a face of every facet of $\Delta$. The elements of $\gamma$ are called the cone points. If $\Delta$ is the iterated cone over $k$ cone points, then $h_{d-k+1}=\cdots=h_d=0$.

Let $\Delta_1$ and $\Delta_2$ be simplicial complexes on disjoint ground sets $E_1$ and $E_2$, the simplicial join $\Delta_1*\Delta_2$ is the complex on the ground set $E_1\cup E_2$ whose faces are unions of faces of $\Delta_1$ and $\Delta_2$.  Simplicial joins of several complexes are defined in the natural straightforward way. The simplicial join of  two spheres is again  a sphere and the join of a sphere and a ball yields another ball. A simplicial complex $\Delta$ is said to be \emph{join irreducible} if it is not equal to the simplicial join of two non-trivial subcomplexes. 

\begin{remark}\label{rem:joins} It well known that the reduced Euler characteristic of the simplicial join of two complexes is the product of the reduced Euler characteristic, that is, $\tilde{\chi}(\Delta_1*\Delta_2) = \tilde{\chi}(\Delta_1)\tilde{\chi}(\Delta_2)$.

\end{remark}

\subsection{PS ear decompositions.}
The \emph{full $d$-simplex} $\Gamma_d$ is the simplicial complex whose faces are all the subsets of a set with $d+1$ elements: it is homeomorphic to a $d$-dimensional ball. The \emph{boundary of the $d$-simplex} $\hat \Gamma_d$ is the set of proper subsets of a set with $d+1$ elements: it is homeomorphic to a $(d-1)$-sphere.
A \emph{$\PS$-sphere} is a join of boundaries of simplices $\hat\Gamma_{d_1}\ast \hat\Gamma_{d_2}\ast\dots \ast\hat\Gamma_{d_k}$. It is homeomorphic to a sphere of dimension $d_1+d_2+\dots +d_k -1$. 
\begin{lemma}\label{lem:hmax}
Let $\Delta$ be any $\PS$-sphere of dimension $d-1$. For every $1\le i \le d$ $\to$ $0\le i \le d$, the following inequality holds: \begin{equation}h_i(\Delta)\leq \binom{d}{i}.\end{equation}
Consequently, $f_{d-1}(\Delta)\leq 2^d$.
\end{lemma}
\begin{proof}
The join operation on simplicial complexes has the effect of multiplying the respective $h$-polynomials. We have that $h(\hat\Gamma_d,t)=1+t+\cdots+t^{d}$, and $h(\hat\Gamma_1^{d},t)=(1+t)^{d}$, where $\hat\Gamma_1^{d}$ is the join of $d$ boundaries of segments.
This implies that, coefficient by coefficient, we have $h(\hat\Gamma_d,t)\leq h(\hat\Gamma_1^{d},t)$.

For a general $\PS$-sphere we have 
$h(\hat\Gamma_{d_1}\ast \hat\Gamma_{d_2}\ast\dots \ast\hat\Gamma_{d_k}, t) = h(\hat\Gamma_{d_1},t)h(\hat\Gamma_{d_2},t)\cdots h(\hat\Gamma_{d_k}, t)\leq h(\hat\Gamma_1^{d_1},t)h(\hat\Gamma_{1}^{d_2},t)\cdots h(\hat\Gamma_{1}^{d_k}, t)=h(\hat\Gamma_1^d,t)$, where $d=d_1+\cdots+d_k$, showing the inequality we wanted. The combinatorially unique maximizer is $\hat\Gamma_1^d$ and it is equal to $\partial \Diamond_d$, the boundary of a $d$-dimensional crosspolytope.\end{proof}

A $\PS$-ball is a complex of the form $\Sigma * \Gamma_\ell$, where $\Sigma$ is a PS-sphere. This is a cone over $\Sigma$ with apex the whole ball $\Gamma_\ell$.
The (topological) boundary of such a PS-ball is the PS-sphere $\Sigma\ast\hat\Gamma_{\ell}$. Notice that, unless $\ell =0$, the vertices of a PS-ball are all in the boundary. In the special case $\ell = 0$ the PS ball has one interior vertex. 

\begin{definition}\label{def:ear}
Let $\Delta$ be a simplicial complex and $K\cong\Sigma * \Gamma_\ell$ a PS-ball with $\dim(\Delta)=\dim(K)$ and such that $ \Delta \cap K = \partial K $. The complex $\Delta'=\Delta\cup K$ is said to be obtained from $\Delta$ by \emph{attaching a $\PS$ ear}.
\end{definition}

\begin{lemma}\label{lem:hvectorear}
Under the conditions of Definition \ref{def:ear} above we have the following relation of $h$-polynomials:
\[
h(\Delta',t) = h(\Delta,t)+t^{l+1}h(\Sigma, t).
\]
\end{lemma}
\begin{proof}
This is the polynomial version of Lemma 3 \cite{Chari} together with the Dehn-Sommerville relations for simplicial spheres. 
\end{proof}

\begin{definition}A $(d-1)$-dimensional simplicial complex $\Delta$ is said to be $\PS$-ear decomposable if there is $k\geq 1$ and a sequence $\Delta_0 \subset \Delta_1 \subset \dots \subset \Delta_{k-1} = \Delta$ of complexes, such that $\Delta_0$ is a PS-sphere and for $0 \le j \le k-2$ the complex $\Delta_{j+1}$ is obtained from $\Delta_{j}$ by attaching a PS-ear. 
\end{definition}
\begin{remark}\label{rem:ear}
By \ref{lem:hvectorear} each time we attach an ear the top Betti number goes up by one and hence if we attach $k-1$ PS ears, the resulting complex has $|\tilde\chi(\Delta)|=k$.
\end{remark}

\subsection{Matroids}
A \emph{matroid} is a pair $M=(E,r)$, where $E$ is a finite set and $r:2^E\to\mathbb{Z}$ is a function on subsets of $E$ such that:
\begin{itemize}
\item[R1] $0\leq r(A)\leq |A|$ for all subsets $A\subset E$.
\item[R2] $r(B)\leq r(A)$ whenever $B\subset A$.
\item[R3] $r(A\cap B)+r(A\cup B)\leq r(A)+r(B)$ for any two subsets $A,B\subset E$.
\end{itemize}
An \emph{independent set} $I\subset E$ is a subset such that $r(I)=|I|$. Independent sets form a simplicial complex denoted by $\II(M)$. A matroid is said to be \emph{connected}, if $\II(M)$ is join irreducible. Maximal independent sets are called \emph{bases} and we denote the set of bases of matroid the matroid by $\BB(M)$. Minimally dependent (that is, not independent) sets are called \emph{circuits}. An element $x$ is called a \emph{loop} if $r(x)=0$. A matroid is said to be loopless if it has no loops. All matroids that we consider in this paper are loop free. An element $x$ is called a \emph{coloop} if $r(E-x)< r(E)$, i.e., it is contained in every basis. A matroid without coloops is said to be \emph{coloop free}. A pair $\{x,y\}$of non-loop elements are \emph{parallel} if $r(\{x,y\})=1$. Being parallel is an equivalence relation so we can talk about parallel classes.
A \emph{simple} matroid is a matroid with $r(A)=|A|$ whenever $|A|\leq 2$, in other words, a matroid with no loops and no parallel pairs.

A \emph{flat} is a subset $F\subset E$ such that $r(F)<r(F\cup\{x\})$ for any $x\notin F$. If we have a total order $<$ on $E$, a \emph{broken circuit} is a circuit with its smallest element removed. A basis is called an \emph{nbc basis} if it does not contain any broken circuit.

An \emph{ordered matroid} $(M,<)$ is a matroid together with an ordering on its ground set. Given an ordered matroid $M$, a basis $B$ and $b\in B$, say that $b$ is \emph{internally passive} if there is $b'<b$ such that $(B\backslash\{b\})\cup\{b'\} \in \BB(M)$, i.e., it can be replaced by a smaller element to obtain another basis of $M$. The set of all internally passive elements of a basis $B$ is denoted by $IP(B)$ and it is called the internally passive set of $B$. 

Let $\text{Int}_<(M)$ be the poset on $\BB(M)$ with the order given by inclusion of internally passive sets. $\text{Int}_<(M)$ is a graded poset with $h_i(\II(M))$ elements of rank $i$. After attaching a maximum element it becomes a graded lattice \cite[Theorem 3.4]{LasV}. As a set system of $E$, $\text{Int}_<(M)$ enjoys the structure of a greedoid \cite{dawson} and \cite[Ex. 7.5]{bjorner}. 

In the paper \cite{bjorner} Bjorner studies three simplicial complexes associated with a matroid $M$. The first one is the independence complex defined above. The other two are defined here: 

\begin{definition}
Let $M=(E,r)$ be a matroid of rank $d$, i.e., $r(E)=d$. We define the following complexes:
\begin{itemize}
\item The \textbf{broken circuit complex} $BC_<(M)$, whenever $(M,<)$ is an ordered matroid, consists of the ground set $E$ with faces given by sets that do not contain broken circuits. It has dimension $d-1$.
\item The \textbf{order complex of the lattice of flats} $\LL(M)$ is the order complex of poset given by flats of $M$ ordered by inclusion (see the precise definitions below). It has dimension $d-2$.
\end{itemize}
\end{definition}

In \cite{bjorner} it is shown that all three complexes are \emph{shellable}, a concept we will not define but only state the consequence we need.  A shellable simplicial complex $\Delta$ of dimension $d-1$ is homotopy equivalent to the wedge product of $k$ spheres of dimension $d-1$, where $k=h_d(\Delta)=|\tilde\chi(\Delta)|$. Hence, its homotopy type depends on just two parameters: $\dim(\Delta)$ and $\tilde\chi(\Delta)$ (or alternatively $h_d(\Delta)$). The latter has the following interpretation in the cases of interest.

\begin{proposition}[\cite{bjorner}]\label{prop:top}
Let $M=(E,r)$ be a matroid of rank $d$, i.e., $r(E)=d$
\begin{itemize}
    \item The broken circuit complex is an iterated cone over a non-contractible space. The number of cone points equals the number of connected components of the matroid . The {\bf reduced} broken circuit complex $\overline{BC}_<(M)$ is the complex that results from removing the cone points of the broken circuit complex. Whenever $M$ is connected, the top $h$ entry of $\overline{BC}_<(M)$ is $\beta(M)$, the \emph{beta invariant} of the matroid.
    \item The top $h$ entry of its order complex is equal to $\mu(\LL(M))$, the \emph{M\"obius number} of the poset of lattice of flats. 
\end{itemize}
\end{proposition}

\begin{definition}[Graphical matroids]
Given a graph $G=(V,E)$, we can define a matroid $M(G)$ on the edge set, $E$, by letting the rank of a subset $A\subset E$ be the size of the largest forest contained in the subgraph induced by $A$. Equivalently, we can define the circuits to be the cycles. 
\end{definition}

\begin{remark}\label{rem:maximizer}
Notice that the maximizer of Lemma \ref{lem:hmax}, $\partial\Diamond_d$, is in fact the independence complex of the graphical matroid given by a path of length $d$ with each edge doubled. See Figure \ref{fig:max}.
\begin{figure}[ht]
\centering
\includegraphics{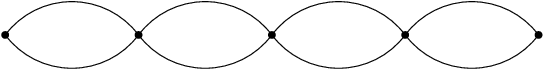}
\caption{Graph whose graphical matroid has independence complex equal to $\partial \Diamond_4$}\label{fig:max}
\end{figure}
\end{remark}

\begin{example}\label{ex:uniform}
Consider the graph $C_{d+1}$ given by a single $(d+1)$-cycle. In the matroid $M(C_{d+1})$ any \emph{proper} subset of $E$ is independent, so the independence complex is $\hat{\Gamma}_d$.
\end{example}

Note that one cannot drop the dimension assumption from Theorem~\ref{thm:main}, since $\tilde\chi(\II(M(C_{d+1})))=1$ for every $d$, as the previous Example show. 

\begin{example}\label{ex:ed}
Consider the graphical matroid $M$ given by the graph in Figure \ref{fig:ed}.
\begin{figure}[ht]
\includegraphics{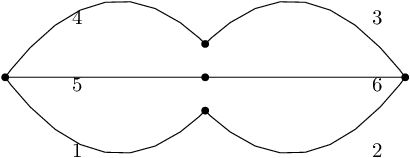}
\caption{Graph in Example \ref{ex:ed}}
\label{fig:ed}
\end{figure}
The circuits are $[1234],[1256],[3456]$ so the broken circuits are $[234],[256],[456]$.\\ 
\textbf{Independent complex:} The bases are $$[1245],[1246],[1235],[1236],[1345],[1346],[1356],[1456],[2345],[2346],[2356],[2456].$$
The $h$-vector is $(1,2,3,4,2)$ so $\II(M)$ complex is homotopy equivalent to the wedge of two spheres of dimension three.\\
\textbf{Broken circuit complex:}
The bases containing no broken circuits are 
$$[1245],[1246],[1235],[1236],[1345],[1346],[1356].$$
The $h$-vector is $(1, 2, 3, 1, 0)$. The zero at the end comes from the fact that we have a cone over the vertex $1$. After removing it, the reduced (see below) broken circuit complex, $\overline{BC}_<(M)$, has $h$-vector $(1,2,3,1)$, so it is homotopy equivalent (but not homeomorphic)  to a two dimensional sphere.
\end{example}


\begin{remark}
We already mentioned in the introduction that it is known that every independence complex arises as a broken circuit complex \cite[Theorem 4.2]{BCC}. Furthermore, the class of independence complexes is \emph{strictly} contained in the class of (reduced) broken circuit complexes. To see this strict containment we go back to Example~\ref{ex:ed}. By \cite[Theorem 3]{Chari} if an independence complex is homotopy equivalent to a sphere, then it is a PS-sphere. The $h$-vector of any PS-sphere is always symmetric so the $h$-vector of the reduced broken circuit complex in Example \ref{ex:ed} is \emph{not} the $h$-vector of any independence complex.
\end{remark}

For any matroid $M$, the independence complex $\II(M)$ is PS-ear decomposable \cite[Theorem 3]{Chari}. This provides a topological difference between independence and broken circuit complexes. Indeed, it follows from the work of Swartz \cite{swartz} that it is \emph{false} for broken circuit complexes.

\subsection{Geometric lattices}\label{sec:geomlattices}

For any matroid $M$ we have a partially ordered set (by inclusion) on the set of flats. These posets are characterized by certain properties; they are precisely the \emph{geometric lattices}. We need some more terminology. 

Let $\PP$ be a finite poset. A poset is bounded if it has a minimum and a maximum, i.e., elements $\hat 0, \hat 1$ , such that $\hat 0\le x \le \hat 1$ for every element of the poset. If $x\neq y$, we say that $x$ covers $y$, denoted $y\precdot x$, if $y\preceq x$ and there is no $z$ different from $x$ and $y$ such that $y\preceq z\preceq x$. An \emph{atom} is an element $x$ such that $\hat{0}\precdot x$. We usually represent a poset through its Hasse diagram, i.e., by drawing an edge between two elements whenever one covers the other.

\begin{definition}
The M\"obius function on a poset $P$ is the unique function $\mu: P\times P\to\Zz$ such that
\[
\mu(x,y)=\begin{cases}
1,\quad x=y.\\
-\sum_{x\preceq z\prec y}\mu(x,z),\quad x\prec y.\\
0,\quad\text{ else.}
\end{cases}
\]
If the poset is bounded then $\mu(P):=\mu(\hat 0,\hat 1).$
\end{definition}

Given two elements $x,y$ we denote by $x\vee y$ their \emph{join}, an element such that $x\preceq z$ and $y\preceq z$ imply $x\vee y\preceq z$. Dually we can define $x\wedge y$ as the \emph{meet}. These operations are binary and associative so it makes sense to talk about the meet or join of any finite subset.

Two elements $x,y$ of $P$ are said to have a \emph{join} if there is an element $x\vee y$ such that if $x\prec z$ and $y\prec z$ if and only if $x\vee y \prec z$. Dually, we can define the \emph{meet} of two elements $x,y$, denoted by $x\wedge y$, whenever it exists. A \emph{lattice} is a poset in which every pair of elements has a join and a meet. As binary operations, joins and meets are associative hence it makes sense to talk about meets and joins of arbitrary subsets of $P$.\\ A chain of length $k$ in a poset $P$ is a collection of distinct elements $x_0\prec x_1 \prec\cdots\prec x_k$. A chain is saturated if the all the relations involved are covering relations. Every lattice is bounded: the meet of all elements is a minimum while the join of all elements is a maximum. A bounded poset is \emph{graded} if every saturated chain starting at $\hat0$ and ending at $\hat 1$ has the same length. The rank of an element $x$ in a graded poset is the length of any saturated chain starting at $\hat 0$ and ending in $x$.
\begin{definition}
A lattice $L$ is said to be geometric if it satisfies the following conditions
\begin{enumerate}
\item It is graded.
\item Its rank function $r$ is semimodular, i.e.,  for every $x,y\in L$ the following inequality holds:
 \begin{equation*}r(x\vee y)+r(x\wedge y)\leq r(x)+r(y).\end{equation*}
\item It is atomistic, i.e., every element is the join of a set of atoms.
\end{enumerate}
\end{definition}

For notational purposes we declare $r(\hat{0})=-1$, so that for instance the atoms have rank equal to zero.
Assigning the poset $\LL(M)$ to each matroid $M$ induces a one-to-one correspondence between geometric lattices and simple matroids \cite[Theorem 1.7.5]{Oxley-book}.  

Every poset $\PP$ gives a simplicial complex $\OO(\PP)$, called the order complex of $\PP$, in the following way: Its elements are the elements of $\PP\backslash\{\hat{0},\hat{1}\}$ and the faces are the chains ordered by inclusion. In a graded poset $\PP$ of rank $r$ all maximal chains in $\PP\backslash\{\hat{0},\hat{1}\}$ have length $r-1$, hence the $\OO(\PP)$ has dimension $r-2$.
As mentioned before, the order complex of a geometric lattice $L$ is shellable. We close this section by providing a description of $|\tilde\chi(\OO(L))|$ following \cite{bjorner}.

Let $m$ be the number of atoms in $L$ and choose an arbitrary bijection between atoms and $[m]$ so we can label atoms with positive integers. Let $E(L)$ be the set of edges of the Haase diagram. Define a labelling $\lambda: E(L)\longrightarrow \mathbb{Z}$ as follows: if $x\succdot y$ then $\lambda(y,x)$ equals the smallest atom $a$ such that $a\preceq x$ but $a\not\preceq y$. A \emph{descending chain} is a chain $\hat{0}=x_0\precdot x_1\precdot\cdots\precdot x_r=\hat{1}$, such that $\lambda(x_{i-1},x_{i})>\lambda(x_{i},x_{i+1})$ for $1\leq i\leq r-1$.

Let $L$ be a geometric lattice. We have that $|\tilde\chi(\OO(L))|=|\mu(\hat{0},\hat{1})|$, where $\mu$ denotes the M\"obius function of $L$ (Philip Hall's theorem \cite[Proposition 3.8.5]{EC1}), and this quantity is also equal to the number of descending chains \cite[Theorem 7.6.4.]{bjorner}.

Notice that this implies that the number of descending chains is independent of the ordering of the atoms.
\begin{example}
Let $M$ be the matroid given by the affine point configuration in the left part of Figure \ref{fig:matroid}. The lattice of flats together with two descending chains are shown right next. We include the whole M\"obius function computation and the order complex which is a graph.
\begin{figure}[ht]
\centering
\includegraphics{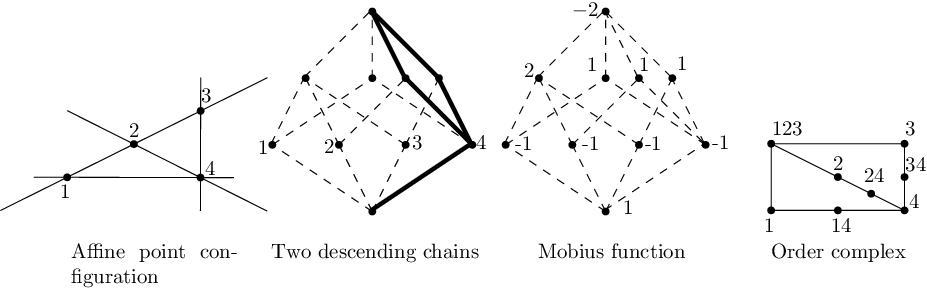}
\caption{An affine point configuration with four numbered atoms and its corresponding lattice of flats.}\label{fig:matroid}
\end{figure}

\end{example}

\section{Independence Complexes}\label{sec:ind}

This section is devoted to various proofs of Theorem~\ref{thm:main}. Quite surprisingly the result is a simple consequence of several known theorems in matroid theory. 

\begin{definition}
Let $\Psi_{d,k}$ be the set of all isomorphism classes of \emph{loopless} matroids $M$ such that $\dim(\II(M))=d-1$ and $|\tilde\chi(\II(M))|=k$.
\end{definition}
Each of the following proofs sheds a light on different aspects of $\Psi_{d,k}$. We begin with a proof using some theorems presented in the survey \cite{bjorner}. These seem to be the oldest family of results that actually suggest the property for matroids. 

\begin{proof}[First proof of Theorem \ref{thm:main}] A matroid is coloop free if and only if $h_d>0$ (see for example \cite[Lemma 3.7]{SamKlee}), so we can assume it is coloop free. By Theorem 7.8.4 and Corollary 7.8.5 in \cite{bjorner}, there is a basis for the homology group $H_{d-1}(\II(M))$ consisting of cycles whose supports are the facets of PS-spheres; furthermore every basis of the complex is in the support of one such cycle. As explained in Chari's original work \cite{Chari}, the number of isomorphism classes of $PS$-spheres of dimension $d$ is counted by the number of integer partitions of $d+1$ and each such $PS$-sphere has at most $2^d$ facets. It follows that the number of bases of $M$ is bounded above by $2^dh_d$. 
\end{proof}

\begin{remark}Notice that the previous bound is far from tight: bases are overcounted and an intricate inclusion/exclusion process is needed. Little is known about the types of spheres in the bases and how they intersect, so 
at present time is not clear to us how to make this argument sharper.

Bj\"orner also shows \cite[Proposition~7.5.3]{bjorner} that if $M$ is connected and has no coloops, then $h_d \ge h_1$. The proof is inductive and uses the Tutte-Polynomial. It is not clear if this is in general tight, but it tells us that if we restrict to connected matroids, then the bounds are different: below we present examples of matroids with $h_1 = h_d + d -1$. 
\end{remark}
Theorem \ref{thm:main} implies the existence of upper bounds for each entry of the $h$-vectors and $f$-vectors of a matroid in terms of its dimension and its Euler characteristic. In \cite{superEd} Swartz proved the following inequalities in a more general set up (see Section \ref{sec:broken}). We now prove the following special case using PS ear decompositions. The advantage of this approach is that we can analyze the equality case.  
\begin{theorem}\label{thm:hbound}
Let $M\in\Psi_{d,k}$ we have the following inequalities:
\begin{enumerate}
\item $h_i(\II(M)) \leq \binom{d}{i} + (k-1)\binom{d-1}{i-1},$ for $0\leq i\leq d$.
\item $f_i(\II(M)) \leq \binom{d}{i+1}2^{i+1} + (k-1)\binom{d-1}{i}2^{i},$ for $-1\leq i\leq d-1$.
\end{enumerate}
Furthermore, these inequalities are tight.
\end{theorem}

\begin{proof}
We begin with the first part. We will use \cite[Theorem 3]{Chari}, i.e., the fact that $\II(M)$ is PS ear decomposable. To begin with, there is a unique $h$-vector maximizer among the PS spheres $\Delta_0$; namely it is the boundary of a $d$-dimensional crosspolytope and its $h$-vector is given by the binomial coefficients (Lemma \ref{lem:hmax}). By Lemma \ref{lem:hvectorear}, together with Lemma \ref{lem:hmax}, the way to attach a PS ear with maximal resulting $h$-vector is by attaching a PS ball $\Sigma*\Gamma_0$ with $\Sigma$ isomorphic to $\partial\Diamond_{d-1}$. We now show that this maximal bound can be attained.

Set $\Delta_0$ to be $\partial\Diamond_d$. Fix a vertex $v\in\Delta_0$ and attach an ear using the PS ball $\Sigma\ast\Gamma_0$, where $\Sigma$ is the link of $v$ (which is isomorphic to $\partial\Diamond_{d-1}$) and $\Gamma_0$ is just a single new vertex.  We can repeat this process $k$ times, always using the same link of the original vertex $v$. The simplicial complex obtained in this way is the independence complex of matroid. Our choice of $\Delta_0$ is the independence complex of the graphical matroid described in Remark \ref{rem:maximizer}. Each ear attachment corresponds to adding parallel elements to a fixed edge. We denote this matroid by $V_{d,k}$.

The second part follows from the fact that $V_{d,k}$ also maximizes each entry of the $f$-vector. This is because by Remark \ref{rem:hpositive} the $f$-vector is a positive combination of the $h$-vector.
\begin{figure}[h]
\centering
\includegraphics{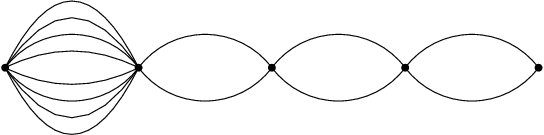}
\caption{The graphical matroid $V_{4,7}$.}
\end{figure}
\end{proof}

\begin{corollary}\label{cor:maximizer}
Among all independence complexes of matroids in $\Psi_{d,k}$, the independence complex of $V_{d,k}$ is the unique maximizer of the $h$-vector.
\end{corollary}

\begin{proof}
If we attach $k-1$ ears then the number of vertices is at most $v+k-1$ where $v$ is the number of vertices of $\Delta_0$. We have $v\leq 2d$ since $\Delta_0$ is a PS sphere and equality only happens for $\Delta_0\cong \partial \Diamond_d$. It follows that the maximum number of vertices among elements of $\Psi_{d,k}$ is $2d+k-1$ with equality only if we start with the boundary of a crosspolytope and each ear attached is of the form $\Sigma\times\Gamma_0$, i.e., every ear introduces a new vertex.\\

The $1$-skeleton of an independence complex is a complete multipartite graph between the parallel classes. Therefore if after attaching all the ears the resulting complex is the independence complex of a matroid, then every new vertex (vertices not belonging to $\Delta_0$) is now part of a single parallel class.
\end{proof}

Now we can give another proof of Theorem \ref{thm:main}.

\begin{proof}[Second Proof of Theorem \ref{thm:main}]
We have $h_d(\II(M))=|\tilde\chi(\II(M))|$, so Theorem \ref{thm:hbound} gives $f_0(\II(M))\leq 2d+h_d(\II(M))-1$. Fixing $h_d(\II(M))$ and $d$ bounds the number of vertices $\II(M)$ can have, whence the result follows.
\end{proof}

In contrast to the case of the Upper Bound Theorem for spheres (see \cite{UBC}), $V_{d,k}$ is the unique maximizer up to isomorphism. However, the matroid $V_{d,k}$ is perhaps not very interesting from the matroid theoretic perspective (for instance the lattice of flats of $V_{d,k}$ is the boolean lattice $B_d$). A relevant variant, which we expect to be harder, is the analogous question over the family of \textbf{simple} matroids.
\begin{question}What is the maximal value of $h_j(\II(M))$ when $M$ ranges over all simple matroids of $\Psi_{d,k}$? Is there a single simple matroid that simultaneously maximizes all the $h$-vector entries ? What if we further restrict to the class of simple connected matroids? 
\end{question}
In light of the above question, we notice that for simple matroids, the number of vertices is strictly less than $2d+ h_d(\II(M)) -1$ which is the tight upper bound for general matroids. 
\begin{corollary}\label{cor:h1max}
 If $M$ is a matroid with $f_0(\II(M)) = 2d+h_d(\II(M)) -1$, then $M$ is isomorphic to $V_{d,k}$. 
\end{corollary}
We now present another proof of the main theorem that may be more suitable for studying the simple case and/or the broken circuit complexes.

\begin{proof}[Third proof of Theorem \ref{thm:main}] Choose an order $<$ on the vertex set of $M$, and consider the poset $\text{Int}_<(M)$. It is graded, the number of elements of degree $i$ is $h_i$ and all the maximal elements are of degree $d$ (since it is a greedoid or a graded lattice minus the top element). Since the elements of the posets are sets ordered by inclusion and graded by cardinality, the number of atoms is at most $d$ times the number of bases of rank $d$ in the poset, in terms of $h$-numbers it means that $h_1 \le dh_d$. 
\end{proof}

\begin{remark} The inequalities obtained from this method are far from tight (Theorem \ref{thm:hbound} gives the stronger inequality $h_1\leq d-1+h_d$). Indeed the equality case would need disjoint bases which cannot happen. The structural properties of $\text{Int}_<(M)$ are quite strong, but barely used. 
\end{remark}

Lastly we present a proof of the main theorem which allows us to say something about the size of $\Psi_{d,k}$.
\begin{proposition}\label{prop:howmany}
Let $d,k>0$ and let $T_{d,k}$ be the number of matroids of rank at most $d$ with at most $k$ bases. Then $2^dkT_{d,k}\ge |\Psi_{d,k}|$.
\end{proposition}

\begin{proof}Given a matroid $M$ and a basis $B$, Corollary~3.5  in \cite{SamKlee} shows that the $h$-polynomial of the independence complex of $M$ can be decomposed as: 
\[ h(\II,x) = \sum_{I} x^{|I|}h(\text{link}_\II(I)|_B, x).\]
The sum is taken over the independent sets $I$ of $M$ that are disjoint from $B$. Lemma~3.8 in \cite{SamKlee} shows that all maximal such $I$ under inclusion, i.e., the bases of the induced matroid on $E\backslash B$, satisfy that $h_{d-|I|}(\text{link}_\II(I)|_B) \not= 0$. It follows that $h_d(M)$ is bounded below by the number of bases of $M|_{E-B}$. This implies that there are at most $k$ maximal bases. Together with the fact that the rank of the restriction is bounded above by $d$, this implies that the number of possible restrictions is finite. The remaining independent sets consist of a subset of $B$ together with an element of the restriction, thus the number of matroids with $h_d = k$ is bounded above by $2^dkT_{d,k}$, where $T_{d,k}$ is the number of matroids of rank at most $d$ with at most $k$ bases. 
\end{proof}

Notice that this provides a fourth proof of Theorem \ref{thm:main}.

\begin{remark} The bounds are far from tight. First of all, it is to be expected that the larger the number of bases of $\Delta_{E\backslash B}$, the fewer ways there are to complete to a matroid. More careful analysis can be carried to replace the power of $2$, but basic asymptotics of binomial coefficients tell us that the replacement is still exponential. An estimate of $T_{d,k}$ is not known, but it seems like estimating it is a more tractable problem. In particular, it is a simple consequence of the exchange axiom that the values stabilize for fixed $k$ and large values of $d$. 
\end{remark}
In general, it follows from \cite[Theorem 3]{Chari} that $|\Psi_{d,1}| = p(d)$, the number of integer partitions of $d$. Consequently, the best kind of formula we can expect for the cardinality of $\Psi_{d,k}$ is asymptotic. It is unclear that the value of $\Psi_{d,k}$ is monotone in either of the parameters. At least the construction of $V_{d,k}$ shows that $\Psi_{d,k}\neq\emptyset$. Using the same ideas we can say a little more.
\begin{lemma}\label{lem:many}
$|\Psi_{d,1}| \le |\Psi_{d,k}|$ for every positive integer $d$.
\end{lemma}
\begin{proof} Since every matroid in $\Psi_{d,1}$ is a PS-sphere, we can choose any vertex $v$ and replicate the construction of $V_{d,k}$ to get an inclusion $\Psi_{d,1}\to \Psi_{d,k}$.
\end{proof}
We have now the ingredients to prove Theorem \ref{thm:howmany}.
\begin{proof}[Proof of Theorem \ref{thm:howmany}]
Combine Proposition \ref{prop:howmany} with Lemma \ref{lem:many}.
\end{proof}

Notice that the previous argument is not strong enough to prove that $\Psi_{d,k} \le \Psi_{d,{k+1}}$ in general (if $d=1$ the number of all such matroids is one). In particular, it would be interesting to find a matroid operation that increases $h_d(\II(M))$ by one in general. The previous construction relies heavily on having a vertex of the independence complex whose link is a sphere. This is, presumably, almost never the case.

\section{Broken Circuit complexes}\label{sec:broken}
Theorem \ref{thm:nbc} is a natural extension of Theorem \ref{thm:main}. Recall that the set of all broken circuit complexes strictly contains the class of independence complexes (see Example \ref{ex:ed}), so any result about $h$-vectors of broken circuit complexes holds for $h$-vectors of independence complexes. In \cite[Theorem 5.4]{superEd} Swartz proved that the inequalities of Theorem \ref{thm:hbound} hold for broken circuit complexes. As a consequence, Theorem \ref{thm:nbc} holds: every $h$-entry is bounded from above by a function of the top $h$-number.

Swartz's proof in \cite{superEd} is an inductive analysis of the Tutte polynomial in which several cases have to be handled separately. Our approaches to the particular case of independence complexes in the previous section has various benefits. We now highlight some of the differences between independence complexes and more general broken circuit complexes.

First we remark that there is no homology basis for $BC_<(M)$ completely extending that in Bj\"orner's theorem. Ziegler \cite[Theorem 1.7]{Ziggy}  constructed a collection of PS-spheres that embed in $\overline{BC} _<(M)$ and generate the top homology group provided that the matroid is connected. Nonetheless, the union of all the constructed spheres does not cover the whole complex, so the first proof of Theorem~\ref{thm:main} cannot be extended using this theorem. For a concrete example, consider the ordered matroid $M$ of Example~\ref{ex:ed}. The $h$-vector of the broken circuit complex is $(1,2,3,1)$, which means that the top homology is one dimensional. If there is a sphere that covers the complex it would have to be equal to it, but that would make the $h$-vector symmetric by the Dehn-Sommerville equations.

Our second proof allowed us to characterize the unique maximizer in Corollary \ref{cor:maximizer}. For general broken circuit complexes we do not know if a similar statement holds. Corollary \ref{cor:h1max} does not hold for broken circuit complexes: Example \ref{ex:ed} shows that $(1,3,2,1)$ and $(1,3,3,1)$ are $h$-vectors of broken circuit complexes. We do not know any bound analogous to Theorem \ref{thm:howmany} for broken circuit complexes.

At the very least, we can prove Theorem \ref{orderId}, which extends the approach of the third proof in Section \ref{sec:ind} using Las Vergnas poset on the set of bases. In order to prove it, we start with a lemma that yields a relation between activities and broken circuits.

\begin{lemma}\label{lem:circuits} Let $(M,<)$ be an ordered matroid and let $C$ be a circuit whose corresponding broken circuit is $\hat C$. If $B$ is any basis with $\hat C \subseteq B$, then $\hat C \subseteq IP(B)$. Furthermore, $\hat C = IP(\hat B)$ for the smallest lexicographic basis $\hat B$ that contains $\hat C$. 
\end{lemma} 
\begin{proof} 
Let $c$ be the element in $C\backslash \hat C$. Since $C\subseteq B\cup\{c\}$ any element $d\in \hat C$ can be replaced by $c$ to obtain a new basis. Since $\hat C$ is a broken circuit, we have $c<d$ and therefore $d\in IP(B)$ as desired. 

If $\hat B$ is the smallest lexicographic basis containing $\hat C$ and $\hat C \subseteq IP(\hat B)$, then equality must hold since the lexicographic order is a shelling order with internally passive sets as restriction sets. 

\end{proof}

\begin{proof}[Proof of Theorem \ref{orderId}]
The Lemma \ref{lem:circuits} implies Theorem~\ref{orderId}: the bases whose internally passive sets are broken circuits form an antichain in $\text{Int}_<(M)$. The order ideal of $\text{Int}_<(M)$ whose minimal non elements are these bases consists exactly of the bases not containing a broken circuit, that is the bases that are facets of the broken cirucuit complex. 
\end{proof}

We finish this section with two finiteness results about broken circuit complexes that are interesting on their own and are crucial steps in the proof of Theorem \ref{thm:latticeStrong} below.

\begin{proposition}\label{prop:latticeStrong}
Let $\mathcal{S}_k$ be the the set of isomorphism classes of broken circuit complexes of simple and coloop free matroids with $k$ facets.
For any $k>0$ the set $\mathcal{S}_k$ is finite.
\end{proposition}

\begin{proof}
The proof has three steps: we first bound the number of vertices $v$ in terms of the rank of the matroid $r$ and $k$, then use this bound to give a lower bound for the number of cone points $t$ in terms of $r$ and $k$, and finally we a argue that the any such coloop free matroid has at most $\log_2(k)$ cone points, yielding an upper bound on the number of vertices.

\begin{enumerate}
    \item Let $\Delta$ be a broken circuit complex with $k$ facets. Since broken circuit complexes are shellable the facet ridge graph $G$ is connected. Fix a basis $B$ in such a broken circuit complex. For any other basis $B'$ the distance in $G$ between $B$ and $B'$ is at most $k$ and this distance bounds from above the number of elements of $B'\backslash B$, i.e., $|B'\backslash B|\leq k$. It follows that $v\leq r+(k-1)k$.
    
    \item Fix a basis $B$. For every non cone point $x$ of $B$ there must be at least one different basis $B'$ such that $x\in B\backslash B'$. Since there are $k-1$ bases different from $B$ and for each one at most $k$ elements are not in $B$, then $r-t\leq (k-1)k$, or equivalently, $r-(k-1)k\leq t$.

    \item Finally, recall from Proposition \ref{prop:top} that $t$ is equal to the number of connected components of the matroid. The lattice of flats of $M$ is the product of the lattices of flats of its connected components $M_i$. By \cite[Proposition 3.8.2]{EC1} we have $\mu(\mathcal{L}(M))=\prod_{i=1}^t\mu(\mathcal{L}(M_i))$. Each factor on the right has absolute value at least two: in the convex ear decomposition of Nyman and Swartz \cite[Section 4]{EdNym} the only lattices of flats with M\"obius number $\pm 1$ correspond to Boolean lattices which are excluded by the coloop free assumption. Hence $ |\mu(\mathcal{L}(M))|\geq 2^t$.
    By \cite[Proposition 7.4.5]{bjorner} the $|\mu(\mathcal{L}(M))|=k$, so we can conclude that $t\leq \log_2(k)$. 
\end{enumerate}

In conclusion, $r-(k-1)k \le t \le \log_2(k)$, which implies $r\le (k-1)k+\log_2(k)$. Thus $v\leq r+(k-1)k\leq 2(k-1)k+\log_2(k)$. In other words we have shown that the number of vertices of any broken circuit complex with at most $k$ facets of a coloopless matroid is $2(k-1)k+\log_2(k)$, so the conclusion follows.
\end{proof}

\begin{proposition}\label{prop:lattice2}
Let $\Delta$ be a simplicial complex and $\mathcal{T}_\Delta$ the set of isomorphism classes of simple ordered matroids $(M,<)$ such that $BC_<(M)$ is isomorphic to $\Delta$. For any $\Delta$ with at least one vertex the set $\mathcal{T}_\Delta$ is finite.
\end{proposition}

\begin{proof}
We will bound the number of vertices $v$ of the independence complex of any such matroid. 
Let $M\in \mathcal{T}_\Delta$ and $C_1, C_2, \dots , C_s$ be the minimal nonfaces of $\Delta$, that is, the broken circuits of $M$. Assume that $C_i\cup x$ and $C_i\cup y$ are circuits of $M$. Pick an arbitrary $z\in C_i$. Note that by the circuit elimination axiom \cite[Lemma 1.1.3]{Oxley-book}, the set $(C_i\cup\{x,y\})\backslash\{z\}$ is a nonface. Since $M$ is simple, $x<y$ are not parallel. Thus there is a circuit of $M$ containing $\{x,y\}$. Such a circuit has to be equal to $C_j\cup\{x\}$ for some $j$ or $C_j\cup\{w\}$ for some $j$ and some other $w$ in the groundset of $M$. In either case $y\in C_j$, and hence a vertex of $\Delta$. Hence the number of vertices of $M$ not in $\Delta$ that extend each broken circuit $C_j$ is at most one, which leads to the inequality $f_0(\II(M)) \le f_0(\Delta) +s$, so the conclusion follows. 
\end{proof}

\section{Order complexes of geometric lattices.}

Recall from Section \ref{sec:geomlattices} that for a geometric lattice $L$ we have that $\OO(L)$ is shellable and that $|\tilde\chi(\OO(L))|=|\mu(\hat{0},\hat{1})|$, hence to classify geometric lattices by their homotopy type is to classify according to their rank and M\"obius number.\\

We begin with a simple argument to show the weaker, rank dependent, part of Theorem \ref{thm:flats}. Let $a(L)$ be the number of atoms of $L$.
\begin{theorem}\label{weakLattices}
The number of isomorphism classes of geometric lattices $L$ with rank $d$ and $|\mu(L)|=k$ is finite.
\end{theorem}
\begin{proof}

We will show that if a rank-$d$ geometric lattice $L$ satisfies $a(L)\geq (k+1)k^{d-1}$, then $|\mu(L)|>k$.

We will proceed by induction on $d$. Let $L$ be a geometric lattice of rank $1$, then $|\mu(L)|=a(L)-1$, the number of atoms, and the base case follows.

Notice that in general if there exist $k+2$ atoms such that their join lies in rank two, then by labeling them with the largest $k+2$ numbers, we can guarantee at least $k+1$ descending chains. So let us assume that no $k+2$ atoms have a join in rank two, i.e., every element in rank two is the join of at most $k+1$ atoms. Fix an atom $x$ and consider the interval $L^x=[x,\hat{1}]$. This interval is a geometric lattice on its own (it corresponds to the matroid obtained by contracting the flat $x$). The atoms of $L^x$ are in bijection with elements of rank two in $L$ above $x$, and as such, they give a partition of the set of atoms of $L$ (other than $x$) by looking at the atoms each of them cover. This means that $k\cdot a(L^x)>a(L)$. Since the rank of $L^x$ is $d-1$, by induction on rank we know that if $a(L)\geq (k+1)k^{d-1}$, then in $a(L^x)\geq (k+1)k^{d-2}$ and therefore there are more than $k$ descending chains. By labeling $x$ with the largest number we can extend each of these chains to descending chains in $L$ to guarantee that $|\mu(L)|>k$. 

\end{proof}

\begin{remark} It should be noted that Swartz and Nyman~\cite{EdNym} proved that the order complex of any geometric lattice admits a convex ear decomposition. This is a decomposition pretty similar to a PS-ear one, except that one is allowed to start with other spheres, and attach other balls (all convex). They use the convex ear decomposition to study flag $h$-numbers, which we intend to do from various points of view in an upcoming project. In their theorem, the combinatorial types of spheres and balls are also prescribed, but form a different class of objects. It seems likely that a different proof of Theorem \ref{weakLattices} can be obtained using these results. 

\end{remark}

The above result looks like a natural extension of Theorem \ref{thm:main}, yet a careful look at Exercise 100(d) in Chapter 3 of \cite{EC1} gives a much stronger result. The level of the problem in the ranking [3-], but unlike most problems in the book, the solution is not written down. To the best of our knowledge, it is not anywhere in the literature, so we include it here for the sake of completeness. 

Recall that $B_d$ is the Boolean lattice of rank $d$, equivalently it is the lattice of flats of the unique matroid of rank $d$ on $d$ elements. 

\begin{theorem}\label{thm:latticeStrong}
Fix a natural number $k$. There exist finitely many geometric lattices $L_1,\cdots,L_m$ such that if $L$ is any finite geometric lattice satisfying $|\mu(L)|=k$ then $L\cong L_i\times B_d$ for some $i,d$.
\end{theorem}
\begin{proof}

By \cite[Proposition 3.8.2]{EC1} we have $\mu(L_i\times B_d)=\mu(L_i)\mu(B_d)=\mu(L_i)(-1)^d$, so the absolute value doesn't change after taking strong product with a boolean lattice.

If $M$ is the matroid associated to $L$, then the matroid $M'$ associated to $L\times B_d$ is obtained by adding $d$ coloops to $M$. Thus it suffices to show that there are finitely many simple coloop free matroids $M$ whose lattice of flats has m\"obius function equal to $k$. 
Assume that $M$ is such a matroid and $\mathcal{L}=\mathcal{L}(M)$ the lattice of flats. By \cite[Proposition 7.4.5]{bjorner} the $|\mu(\mathcal{L})|$ equals the number of facets of $BC_<(M)$ for any order $<$ on the ground set of $M$. Combining Proposition \ref{prop:latticeStrong} and Proposition \ref{prop:lattice2} the conclusion follows.
\end{proof}

\begin{remark} We notice that the proof of the previous theorem is far from sharp. In general, a matroid has many different broken circuit complexes that vary as the order changes. 
\end{remark}
As a corollary of Theorem \ref{thm:latticeStrong} we can drop the rank condition in the statement of Theorem \ref{thm:main} in the case of cosimple matroids.

\begin{corollary}\label{cor:cosimple} Let $k$ be a positive integer. There are finitely many isomorphism classes of cosimple loop free matroids $M$ whose independence complex satisfies $|\tilde{\chi}(\II(M))|=k$. 
\end{corollary}

\begin{proof}
Assume that $M$ satisfies the hypothesis. Then the dual matroid $M^*$ is a coloop free simple matroid. By \cite[Proposition 7.4.7]{bjorner} we have $k=|\tilde{\chi}(\II(M))|=|\mu(\LL(M^*))|$. Since $M^*$ has no coloops, the lattice $\LL(M^*)$ has no boolean factor and it is therefore in the finite list provided by Theorem \ref{thm:latticeStrong}. The result follows from the bijective correspondence between simple matroids and geometric lattices.
\end{proof}

%
%
\section{Further Questions}
The matroids constructed in Lemma \ref{lem:many} are all non simple. The following question may inspire interesting constructions of matroids. 
\begin{question} Let $d,k$ be two positive integers. Is there a simple rank-$d$ matroid $M$ with $h_d(\II(M))=k$?
\end{question}

Of special interest is the case of $k=2$. We already saw that if $d=2$, then the answer is no. However, starting with $d=3$ such a matroid always exists.
\begin{theorem} If $d\ge 3$ there exists a simple rank $d$ matroid $M$ with $h_d(\II(M)) = 2$. 
\end{theorem}
 \begin{proof}Consider the PS-sphere $\Delta_0=\hat\Gamma_{d-1}*\hat\Gamma_1$. Attach the ear $\hat\Gamma_{d-2} *\Gamma_1$ identifying the vertices of $\hat\Gamma_{d-2}$ with any set of vertices of $\hat\Gamma_{d-1}$ to obtain a new complex $\Delta_1$. The complex $\Delta_1$ is the independence complex of a matroid $M$ since every induced subcomplex is pure. The graph of $\Delta_1$ is complete because the only missing edge of $\Delta_0$ is added when we the ear is attached. This implies that the matroid $M$ is simple since parallel elements form the missing edges of the independence complex. Finally, by Remark \ref{rem:ear} we have $h_d(\II(M))=2$.
\end{proof}

It is still not clear how many such matroids there are. It seems that $\hat\Psi_{d,1}$ can be embedded in $\hat\Psi_{d,2}$ by similar tricks, but we may note that the PS-ear decomposition is not necessarily unique and the results have to be dealt with carefully.

Pushing the question a bit further leads us to wonder about new techniques to construct matroids by keeping the dimension and and changing homology. The methods we have so far feel adhoc.  

\begin{problem} Given a rank $d$ matroid $M$ that is not a cone, construct a rank $d$ matroid $\hat M$ with $h_d(\hat M) = h_d(M) + 1$. A variant with $h_d(\hat M)= h_d(M)+c$ for a fixed constant $c$ that may depend on $d$ would also be of interest. 
\end{problem} 
\begin{question} Given a matroid $M$ is there a subset $U$ of the set of bases of $M$, that is the set of bases of a matroid $\overline M$ such that $h_d(\overline M) = h_d(M) - 1$? 
\end{question}
Attaching ears sometimes turns an independence complex into a non-independence complex. We provide a conjecture along the lines of these results. 

\begin{question} Assume that $\Delta$ is the independence complex of a matroid and let $\Delta'$ be a complex obtained from $\Delta$ by attaching a PS-ear that does not introduce a new vertex. Under which conditions is $\Delta'$ the independence complex of a matroid? 
\end{question}

Notice that if the PS-ball is of the form $\Sigma*\Gamma_i$ (with $i>1$), then all that is needed is that all the induced subcomplexes of vertex sets containing all the vertices of $\Gamma_i$ are pure. 

In contrast if an ear is attached and a new vertex is introduced, then the resulting complex can potentially be a matroid if and only if it is connected to all vertices not parallel to it. That seems to be a rare property: there has to be a parallel class whose complementary set of vertices induces a PS-sphere. 

The database of matroids in \cite{Database} list matroids according to rank and number of vertices. The classification allows the user access to lists of matroids with up to nine elements, and matroids with small ranks and a few more elements. The data base considers cases of simple and non simple matroids and has been quite useful in testing conjectures and finding examples of interesting matroids. 
\begin{question}
Is there an algorithm that generates all matroids of a given rank and topology efficiently for some (hopefully not very small) parameters?
\end{question}

A brute force approach can be worked from the already existing database of matroids. From the fact that $f_0(\II(M)) \le 2d + k -1$ we can extract all such matroids for some small values of $d$. In rank $3$ all the matroids with $h_d \le 5$ are contained in the database. For rank $4$ all simple matroids with $h_d \le 2$ are also in the database. This is, however, not interesting enough.

In the case of geometric lattices several invariants besides the $h$-vector of the order complex are of interest. For instance, it may be of interest to bound the Whitney numbers (of both kinds) and the flag $h$-vector in terms of the prescribed topology. We finalize by posing two questions. 

\begin{question} Given $k>0$, what is the largest rank of a geometric lattice $\LL$ that does not contain a factor of $B_n$ for any $n$ and such that $|\mu(\LL)|=k$? 
\end{question}

\begin{question}
Does Corollary \ref{cor:cosimple} extend to the class of Broken Circuit complexes, i.e is it true that the rank conditions in Theorem \ref{thm:nbc} can be dropped for cosimple matroids?
\end{question}
We remark that for connected matroids $M$ we have that $h_{d-1}(\overline{BC_<(M)})$ equals the \emph{beta invariant} of the matroid $\beta(M)$. In \cite{oncrapo} J.Oxley classified all matroids whose beta invariant is between $1$ and $4$ in terms of excluded minors. It would be interesting to compare Theorem \ref{thm:nbc} to that classification. In particular it may be possible that the ideas of the classification can be used to produce a database of matroids stratified by rank and beta invariant. \\

\noindent{\bf Acknowledgements:} We would like to thank Richard Stanley for interesting conversations and for pointing out the reference in his book to Theorem~\ref{thm:latticeStrong}. Thanks to Ed Swartz for reminding us of Example~\ref{ex:ed}. An anonymous referee pointed out the connections between geometric lattices and cosimple matroids that inspired Corollary \ref{cor:cosimple}. We are specially indebted to Isabella Novik for various interesting conversations and helpful suggestions on preliminary versions. We are grateful to the University of Washington and University of Kansas where parts of this project were carried out. The second named author also thanks the University of Miami where he was employed when most of the project was carried out. This project was completed while both authors were members of the Max-Planck Institute for Mathematics in the Sciences.  

\bibliographystyle{alpha}
\bibliography{BIBLIO}
\end{document}